\documentclass[final,leqno,letterpaper]{etna}

\newcommand{\bfD}{{\bf D}}

\newcommand{\bfH}{{\bf H}}
\newcommand{\bfI}{{\bf I}}
\newcommand{\bfJ}{{\bf J}}

\newcommand{\bfM}{{\bf M}}

\newcommand{\bfT}{{\bf T}}
\newcommand{\bfU}{{\bf U}}
\newcommand{\bfV}{{\bf V}}

\newcommand{\bfX}{{\bf X}}

\newcommand{\bfZ}{{\bf Z}}

\newcommand{\bfa}{{\bf a}}
\newcommand{\bfb}{{\bf b}}
\newcommand{\bfc}{{\bf c}}

\newcommand{\bfx}{{\bf x}}
\newcommand{\bfy}{ {\bf y}}

\newcommand{\bfp}{{\bf p}}

\newcommand{\bfv}{{\bf v}}

\newcommand{\R}{\ensuremath{\mathds{R}}}

\usepackage{ifthen,algorithm,algorithmic}
\usepackage{amssymb,amsfonts,amsmath,latexsym,dsfont}
\usepackage{tikz,pgflibraryplotmarks}
\usepackage{pgfplots}
\usepackage{standalone}
\usepackage{graphicx}
\usepackage{mathrsfs}
\usepackage{amsmath}

\usepackage{hyperref}
\usepackage{cleveref}
\usepackage{lineno}

\usepackage{bbm}

\DeclareMathOperator*{\argmin}{arg\,min}

\usepackage{multirow}
\usepackage{multicol}

\graphicspath{{figures/}{./}{./figures/}}

\newdimen\iwidth
\newdimen\iheight

\setbibdata{1}{xx}{46}{2023} 

\hypersetup{%
    pdftitle={LSEMINK: A Modified Newton-Krylov Method for Log-Sum-Exp Minimization},
    pdfauthor={Kelvin Kan, James G. Nagy, and Lars Ruthotto},
    pdfkeywords={log-sum-exp minimization, Newton-Krylov method, modified Newton method, machine learning, geometric programming}
    }

\title{LSEMINK: A Modified Newton-Krylov Method for Log-Sum-Exp Minimization\thanks{%
Received... Accepted... Published online on... Recommended by....
}}

\author{Kelvin Kan\footnotemark[2] 
        , James G. Nagy\footnotemark[3]
        ,\and Lars Ruthotto\footnotemark[3]}

\shorttitle{LSEMINK: A Modified Newton-Krylov Method} 
\shortauthor{K.~KAN, J.~Nagy AND L.~Ruthotto}

\begin{document}

\maketitle

\renewcommand{\thefootnote}{\fnsymbol{footnote}}

\footnotetext[2]{Department of Mathematics, Emory University, USA (kelvin.kan@emory.edu)}
\footnotetext[3]{Departments of Mathematics and Computer Science, Emory University, USA (jnagy@emory.edu, lruthotto@emory.edu)}

\begin{abstract}
This paper introduces LSEMINK, an effective modified Newton-Krylov algorithm geared toward minimizing the log-sum-exp function for a linear model. 
Problems of this kind arise commonly, for example, in geometric programming and multinomial logistic regression. 
Although the log-sum-exp function is smooth and convex, standard line search Newton-type methods can become inefficient because the quadratic approximation of the objective function can be unbounded from below. 
To circumvent this, LSEMINK modifies the Hessian by adding a shift in the row space of the linear model. We show that the shift renders the quadratic approximation to be bounded from below and that the overall scheme converges to a global minimizer under mild assumptions. 
Our convergence proof also shows that all iterates are in the row space of the linear model, which can be attractive when the model parameters do not have an intuitive meaning, as is common in machine learning.
Since LSEMINK uses a Krylov subspace method to compute the search direction, it only requires matrix-vector products with the linear model, which is critical for large-scale problems.
Our numerical experiments on image classification and geometric programming illustrate that LSEMINK considerably reduces the time-to-solution and increases the scalability compared to geometric programming and natural gradient descent approaches. It has significantly faster initial convergence than standard Newton-Krylov methods, which is particularly attractive in applications like machine learning. In addition, LSEMINK is more robust to ill-conditioning arising from the nonsmoothness of the problem. We share our MATLAB implementation at~\url{https://github.com/KelvinKan/LSEMINK}.
\end{abstract}



\section{Introduction}\label{sec:intro}
We consider minimization problems of the form
\begin{equation}\label{eq:general_f}
    \min_{\bfx \in \R^n} f(\bfx) = \sum_{k=1}^N w^{(k)} \left[ g^{(k)}(\bfx) - {\bfc^{(k)}}^\top \bfJ^{(k)} \bfx \right],
\end{equation}
where
\begin{equation*}
    g^{(k)}(\bfx):=\log \left({\bf 1}_m^\top \exp({\bfJ^{(k)} \bfx + \bfb^{(k)}}) \right)
\end{equation*}
is the log-sum-exp function for a linear model defined by $\bfJ^{(k)} \in \mathbb{R}^{m \times n}$ and $\bfb^{(k)} \in \mathbb{R}^m$, $\bfc^{(k)} \in \mathbb{R}^m$,  ${\bf 1}_m \in \mathbb{R}^{m}$ is a vector of all ones, $w^{(k)}$'s are weights, and $N$ is the number of linear models.
Problem~\eqref{eq:general_f} arises commonly in machine learning and optimization. For example, multinomial logistic regression (MLR) in classification problems~\cite{fung2019admm,newman2021train,kan2021pnkh} is formulated as~\eqref{eq:general_f}. In geometric programming~\cite{tseng2015milp,xi2020log,zhan2018accelerated}, a non-convex problem can be convexified through a reformulation to the form~\eqref{eq:general_f}. The log-sum-exp function itself also has extensive applications in machine learning. For instance, it can serve as a smooth approximation to the element-wise maximum function~\cite{ghodousian2022log,nielsen2016guaranteed}, where smoothness is desirable in model design since gradient-based optimizers are commonly used. Moreover, the log-sum-exp function is closely related to widely used softmax and entropy functions. For instance, the dual to an entropy maximization problem is a log-sum-exp minimization problem~\cite[Example~5.5]{boyd2004convex}, and the gradient of the log-sum-exp function is the softmax function~\cite{gao2017properties}.

Despite the smoothness and convexity of the log-sum-exp function, a standard implementation of line search Newton-type methods can be problematic. To realize this, note that the gradient and Hessian of the log-sum-exp function are given by
\begin{align*}
\begin{split}
    \nabla f(\bfx) = \sum_{k=1}^N w^{(k)} {\bfJ^{(k)}}^\top (\bfp^{(k)} - \bfc^{(k)}), \quad &\text{and} \quad \nabla^2 f(\bfx) = \sum_{k=1}^N w^{(k)} {\bfJ^{(k)}}^\top \bfH^{(k)} \bfJ^{(k)}, \\
    \text{with} \quad \bfp^{(k)} = \frac{\exp({\bfJ^{(k)} \bfx + \bfb^{(k)}})}{{\bf 1}_m^\top \exp({\bfJ^{(k)} \bfx + \bfb^{(k)}})}, \quad &\text{and} \quad \bfH^{(k)} = {\rm diag}(\bfp^{(k)}) - \bfp^{(k)} {\bfp^{(k)}}^\top.
\end{split}
\end{align*}
The Hessian is positive semi-definite and rank-deficient because the null space of the $\bfH^{(k)}$'s contains ${\bf 1}_m$. Even more problematic is that when $\bfp^{(k)}$'s are close to a standard basis vector (which, for example, commonly occurs in MLR), the Hessian is close to the zero matrix even when the gradient is non-zero. In Newton's method, this means that the local quadratic approximation can be unbounded from below. To be precise, it is unbounded from below if and only if the gradient is not in the column space of the Hessian~\cite[Exercise~2.19]{beck2014introduction}. 

Disciplined convex programming (DCP) packages (e.g., CVX~\cite{grant2008cvx}) can reliably solve the log-sum-exp minimization problem through a reformulation. For instance, CVX first formulates the problem using exponential cones~\cite[Section~5.2.6]{aps2020mosek} and applies backend solvers to solve the resulting problem directly (e.g., MOSEK~\cite{aps2019mosek}) or through successive polynomial approximation (e.g., SPDT3~\cite{tutuncu2003solving} and SeDuMi~\cite{sturm1999using}). However, this approach can be computationally demanding as the number of conic constraints scales with the product of the number of rows in the linear models and the number of linear models. For instance, CVX did not complete the image classification experiments for the whole dataset in~\Cref{subsec:classification} on a standard laptop in thirty minutes, while LSEMINK finishes on the same hardware in thirty seconds. Furthermore, the formulation relies on access to the elements of the $\bfJ^{(k)}$'s; i.e., this approach is not applicable in a matrix-free setting where $\bfJ^{(k)}$'s are not built explicitly, and only routines for performing matrix-vector products are provided.

Tikhonov regularization~\cite{engl1996regularization,goodfellow2016deep,hansen1998rank}, which adds $\frac{\alpha}{2} \| \bfx \|_2^2$ with $\alpha>0$ to the objective function,  avoids the cost of reformulation and alleviates the convergence issues with Newton-type methods. 
The regularization shifts the Hessian by $\alpha \bfI$ and renders it positive definite, where $\bfI$ is the identity matrix. Nonetheless, Tikhonov regularization introduces a bias and consequently changes the optimal solution. The regularization parameter $\alpha$ has to be chosen judiciously -- a large $\alpha$ renders the problem easier to solve and produces a more regular solution but introduces more bias. In addition, one cannot use effective parameter selection algorithms~\cite{golub1979,chung2008,calvetti1999,vogel2002,kan2020avoiding} for linear problems due to the nonlinearity of the log-sum-exp function. On the other hand, first-order methods like gradient descent~\cite{boyd2004convex,nocedal2006}, or AdaGrad~\cite{duchi2011adaptive}, which do not use the Hessian matrix, can avoid the problem. However, their convergence is inferior to methods that utilize curvature information~\cite{dunn1980}.

Modified Newton-type methods effectively tackle problems with rank-deficient or indefinite Hessians and do not introduce bias. The idea is to add a shift to the Hessian so that at the $i$th iteration, the scheme solves
\begin{equation}\label{eq:logsumexp_shift_general}
    \min_{\bfx} \frac{1}{2} (\bfx - \bfx_i)^\top (\nabla^2 f(\bfx_i) + \beta_i \bfM_i) (\bfx - \bfx_i) + \nabla f (\bfx_i)^\top (\bfx - \bfx_i),
\end{equation}
where $\beta_i$ is a parameter, and the shift $\bfM_i$ renders the Hessian to be sufficiently positive definite. The quadratic approximation is bounded from below since the modified Hessian is positive definite. Hence the convergence issues are avoided.  The effect of the Hessian shift is reminiscent of the Tikhonov regularization approach. Indeed, the scheme is sometimes called a Tikhonov-regularized Newton update~\cite[Chapter~3.3]{parikh2014proximal}. However, the key conceptual difference between~\eqref{eq:logsumexp_shift_general} and Tikhonov regularization is that the former does not introduce any bias to the problem~\cite{vidal2022taming}, i.e., the optimal solution to the problem is independent of $\beta_i$'s. There are different ways of defining $\bfM_i$. For instance, $\bfM_i$ is spanned by some of the eigenvectors of the Hessian~\cite{greenstadt1967relative,nocedal2006}, or is a modification to the factorization of the Hessian~\cite{gill1974newton,more1979use,nash1984newton}. However, the computations needed for these approaches are intractable for large-scale problems commonly arising in machine learning. A simple and computationally feasible approach is to set $\bfM_i$ as the identity matrix~\cite{levenberg1944method,marquardt1963algorithm,parikh2014proximal}, which will be used as a comparing method in our numerical experiments.

In this paper, we propose LSEMINK, a novel modified Newton-Krylov method that circumvents the drawbacks outlined above. The main novelty in our method is the Hessian shift $\bfM_i=\sum_{k=1}^{N} w^{(k)} {\bfJ^{(k)}}^\top \bfJ^{(k)}$. This generates an update in the row space of the linear model, as compared to the aforementioned modified Newton-type methods, which returns an update in the parameter space of the linear model (i.e., the $\bfx$ space). This property is preferable in machine learning applications since model parameters often do not have an intuitive meaning, while the row space of the linear model contains  interpretable data features. Note that standard convergence guarantees (e.g., ~\cite[Chapter~6.2]{nocedal2006}), which often require positive definiteness of the modified Hessian, do not apply to our method since our modified Hessian can be rank-deficient. We show that the quadratic approximation is bounded from below, and the overall scheme provably converges to a global minimum. Since a Krylov subspace method is applied to approximately solve~\eqref{eq:logsumexp_shift_general} to obtain the next iterate, LSEMINK is suitable for large-scale problems where the linear models are expensive to build and are only available through matrix-vector multiplications. Our numerical experiments on image classification and geometric programming illustrate that LSEMINK considerably reduces the time-to-solution and increases the scalability compared to DCP and natural gradient descent and has significantly faster initial convergence than standard Newton-Krylov methods.

This paper is organized as follows. In~\Cref{sec:proposed_method}, we describe the proposed LSEMINK. In~\Cref{sec:proof}, we provide a global convergence guarantee. In~\Cref{sec:experiments}, we demonstrate the effectiveness of LSEMINK using two numerical experiments motivated by geometric programming and image classification, respectively. We finally conclude the paper in~\Cref{sec:conclusion}.

\section{LSEMINK}\label{sec:proposed_method}
We propose LSEMINK, a modified Newton-Krylov method geared toward log-sum-exp minimization problems of the form~\eqref{eq:general_f}. 
At the $i$th iteration, we first consider the quadratic approximation~\eqref{eq:logsumexp_shift_general} with $ \bfM_i=\sum_{k=1}^N w^{(k)} {\bfJ^{(k)}}^\top \bfJ^{(k)}$. That is,
\begin{align}
\hspace{-0.7em}
\begin{split}\label{eq:logsumexp_shift}
    \min_{\bfx} q_i(\bfx) &= \frac{1}{2} (\bfx - \bfx_i)^\top \left(\nabla^2 f(\bfx_i) + \beta_i \sum_{k=1}^N w^{(k)} {\bfJ^{(k)}}^\top \bfJ^{(k)} \right) (\bfx - \bfx_i) + \nabla f (\bfx_i)^\top (\bfx - \bfx_i) \\
    &=  \frac{1}{2} 
     (\bfx - \bfx_i)^\top \left[ \sum_{k=1}^N w^{(k)} \left( {\bfJ^{(k)}}^\top (\bfH^{(k)}_i + \beta_i \bfI) \bfJ^{(k)} \right) \right] (\bfx - \bfx_i) + \nabla f (\bfx_i)^\top (\bfx - \bfx_i),
\end{split}
\end{align}
whose minimizer is given by $\bfx_i + \Delta \bfx_i$, where $\Delta \bfx_i$ solves the Newton equation
\begin{equation}\label{eq:Newton_eq}
    \nabla^2 q_i(\bfx_i) \Delta \bfx_i = - \nabla q_i (\bfx_i),
\end{equation}
and $\bfH^{(k)}_i$ is $\bfH^{(k)}$ evaluated at $\bfx_i$. It is important to note that the Hessian shift in~\eqref{eq:logsumexp_shift} is different from the typical modified Newton approaches (e.g. eigenvalue modification~\cite{greenstadt1967relative,nocedal2006}, identity matrix~\cite{levenberg1944method,marquardt1963algorithm,parikh2014proximal}, or modification to the factorization of the Hessian~\cite{gill1974newton,more1979use,nash1984newton}) which seek to obtain a positive definite Hessian and lead to an update in the parameter space of the linear model (i.e. the $\bfx$ space).
Instead, it generates an update direction in the row space of the linear models. This is preferable especially in machine learning applications because model parameters often do not have intuitive meaning while the row space of the linear models contains data features and is explicable. Although the Hessian of~\eqref{eq:logsumexp_shift} is rank-deficient especially when the linear models are over-parametrized (i.e. $\bfJ^{(k)}$'s have more columns than rows), it is positive definite in the row space of the linear model. Consequently, the quadratic approximation is bounded from below, and the overall scheme provably converges to a global minimum; see~\Cref{sec:proof} for a detailed derivation. 

An alternative formulation for~\eqref{eq:logsumexp_shift} is
\begin{equation*}
    \min_{\bfx}  \frac{1}{2} (\bfx - \bfx_i)^\top \nabla^2 f(\bfx_i) (\bfx - \bfx_i) + \nabla f (\bfx_i)^\top (\bfx - \bfx_i) + \frac{\beta_i}{2} \sum_{k=1}^N w^{(k)} \| \bfJ^{(k)} (\bfx - \bfx_i)\|_2^2,
\end{equation*}
which can be interpreted as a Newton scheme with a proximal term acting on the row space of $\bfJ^{(k)}$'s. This formulation shows that $\beta_i$ controls the step size in a nonlinear line search arc. To be precise, $\beta_i=0$ and $\infty$ correspond to a Newton update with step size $1$ and $0$, respectively, and the update is given nonlinearly for $0<\beta_i<\infty$. The formulation also shows that our proposed scheme bears similarity to $L^2$ natural gradient descent (NGD) methods~\cite{pascanu2013revisiting,NurbekyanEtAl2022} which use the same proximal term. Nonetheless, unlike our approach, $L^2$ NGD methods generally do not directly incorporate Hessian information into its search direction and approximate curvature information using only the linear model. 

The crucial difference between the proximal term and Tikhonov regularization is that the former does not introduce any bias~\cite{parikh2014proximal,vidal2022taming}; i.e., the optimal solution is independent of $\beta_i$. Another advantage is that Tikhonov regularization requires parameter tuning, which is commonly done using a grid search for nonlinear problems like~\eqref{eq:general_f}, while in our proposed method $\beta_i$'s are automatically selected by a backtracking Armijo line search scheme. The proposed scheme can also be perceived as a proximal point algorithm acting on the second-order approximation~\cite{parikh2014proximal}.

We compute the update direction $\Delta \bfx_i$ by approximately solving the Newton equation~\eqref{eq:Newton_eq} using a Krylov subspace method (e.g., conjugate gradient method~\cite{nocedal2006,boyd2004convex}) and obtain the next iterate $\bfx_{i+1} = \bfx_{i} + \Delta \bfx_{i}$. In particular, the Krylov subspace is given by
\begin{align}\label{eq:Krylov}
\begin{split}
    & \quad \; \mathcal{K}_r(\nabla^2 q_i(\bfx_i), \nabla q_i (\bfx_i)) \\
    &= \mathcal{K}_r \left(\sum_{k=1}^N w^{(k)} \left( {\bfJ^{(k)}}^\top (\bfH^{(k)}_i + \beta_i \bfI) \bfJ^{(k)} \right), \sum_{k=1}^N w^{(k)} {\bfJ^{(k)}}^\top (\bfp_i^{(k)} - \bfc^{(k)}) \right),
\end{split}
\end{align}
where $r$ is the dimension of the Krylov subspace and $\bfp_i^{(k)}$ is $\bfp^{(k)}$ evaluated at $\bfx_i$.
Since the Krylov subspace method only requires routines to perform Hessian-vector multiplications, LSEMINK is applicable to large-scale problems commonly arising in machine learning applications where the linear models are only available through matrix-vector products. 
An outline of the implementation of LSEMINK is presented in~\Cref{alg:outline}.


LSEMINK has significantly faster initial convergence compared with standard Newton-Krylov solvers. This is particularly attractive in applications that do not require high accuracy, e.g., image classification. LSEMINK also considerably reduces the time-to-solution and has better scalability compared to geometric programming and natural gradient descent approaches. It avoids the respective drawbacks of the solvers outlined in~\Cref{sec:intro}. Moreover, it is more robust to ill-conditioning arising from the nonsmoothness of the problem; see~\Cref{sec:experiments} for numerical experiments. We provide a MATLAB implementation at~\url{https://github.com/KelvinKan/LSEMINK}. The implementation is easy to experiment with, as it only requires minimal knowledge and input from the user.

\begin{algorithm}[t]
\begin{algorithmic}[1]
 \STATE Inputs: Linear models $\bfx \mapsto \bfJ^{(k)}\bfx$, $\bfx \mapsto {\bfJ^{(k)}}^\top \bfx$, $\bfb^{(k)}$, $\bfc^{(k)}$, and weights $w^{(k)}$ for $k=1,2,...,N$. Initial guess $\bfx_0$, initial $\beta_0$.
 \STATE Inputs: Tolerances xtol, gtol for Newton iterations. Tolerances ktol and kmaxiter for the Krylov subspace method. Line search parameter $\gamma \in (0,1)$.
  \FOR{$i=0, 1,2,\ldots$}
    \STATE compute $f(\bfx_i)$, $\nabla f(\bfx_i)$ and build routines for performing $ \bfv \mapsto \nabla^2 f(\bfx_i) \bfv $
    \FOR{$j=0,1,2,\ldots$}
    \STATE compute $\Delta \bfx_i$ by applying Krylov-subspace methods to 
    approximately solve $\nabla^2 q_i(\bfx_i)  \Delta \bfx_i = - \nabla q_i(\bfx_i)$ with the current $\beta_i$ and Krylov subspace $\mathcal{K}_r(\nabla^2 q_i (\bfx_i), \nabla q_i(\bfx_i))$ until the relative residue drops below ktol or number of iterations exceeds kmaxiter    
    \IF{$f(\bfx_i + \Delta \bfx_i) < f(\bfx_i) + \gamma \nabla f(\bfx_i)^\top \Delta \bfx_i$}
    \STATE set $\bfx_{i+1}=\bfx_{i} + \Delta \bfx_{i}$ and break
    \ELSE
    \STATE set $\beta_{i} = 2\beta_{i}$
    \ENDIF
    \ENDFOR
    \IF{ $\| \bfx_{i+1} - \bfx_{i} \|_2 / \| \bfx_{i} \|_2 <$ xtol or $\| \nabla f(\bfx_{i+1})\|_2 <$ gtol}
    \STATE break
    \ENDIF
    \IF{$j=0$}
    \STATE set $\beta_{i+1} = 0.5*\beta_{i}$
    \ELSE
    \STATE set $\beta_{i+1} = \beta_{i}$
    \ENDIF
 \ENDFOR
 \STATE Output: approximate solution $\bfx_{i+1}$. 
\end{algorithmic}
\caption{Outline of LSEMINK for solving~\eqref{eq:general_f}}
\label{alg:outline}
\end{algorithm}

\section{Proof of Global Convergence}\label{sec:proof}
In this section, we prove the global convergence of the proposed LSEMINK. It is noteworthy that existing convergence results cannot be directly applied due to the rank-deficiency of our modified Hessian. For instance, it is assumed in~\cite[Chapter~6.2]{nocedal2006} that the modified Hessian is positive definite and has a bounded condition number. Our proof is modified from the approach in~\cite{lee2014proximal}, which studies proximal Newton-type methods for composite functions. 

We first state the main theorem.
\begin{theorem}\label{thm:main}
Assume that $f$ is defined in~\eqref{eq:general_f},
and $\inf\limits_{\bfx} f(\bfx)$ is attained in $\mathbb{R}$, then the sequence $\{ \bfx_i \}_i$ generated by LSEMINK converges to a global minimum regardless of the choice of initial guess $\bfx_0$.
\end{theorem}
We note that \Cref{thm:main} also applies to the case where the Newton equation~\eqref{eq:Newton_eq} is solved exactly. In the following, we will first discuss some properties of LSEMINK. We will then state and prove four lemmas which will aid the proof of~\Cref{thm:main}.

For simplicity of exposition and without loss of generality, in this section, we drop the superscript and focus on the case with only one linear model defined by $\bfJ$, $\bfb$, and $\bfc$, and the weight $w=1$. In this case, the Krylov subspace in~\eqref{eq:Krylov} becomes
\begin{equation}\label{eq:sim_Krylov}
    \mathcal{K}_r(\nabla^2 q_i(\bfx_i), \nabla q_i (\bfx_i)) = \mathcal{K}_r(\bfJ^\top (\bfH_i + \beta_i \bfI) \bfJ, \bfJ^\top (\bfp_i - \bfc)).
\end{equation}
We note that our proof can be straightforwardly extended to the general case by setting
\begin{align*}
    \bfJ = [\bfJ^{(1)}; ...; \bfJ^{(N)}], \quad & \bfc = [w^{(1)}\bfc^{(1)};  ...; w^{(N)}\bfc^{(N)}], \\
    \bfp_i = [w^{(1)}\bfp_i^{(1)}; ...; w^{(N)}\bfp_i^{(N)}], \quad \text{and} \quad & \bfH_i = {\rm blkdiag}(w^{(1)}\bfH_i^{(1)}, ..., w^{(N)}\bfH_i^{(N)}),
\end{align*}
where $\rm blkdiag$ denotes a block diagonal matrix.

Recall that the Krylov subspace in~\eqref{eq:sim_Krylov} is constructed to approximately solve the Newton equation and obtain the update direction $\Delta \bfx_i$. This is equivalent to building a rank-$r$ approximation $\nabla^2 q_i(\bfx_i) \approx \bfV_i \bfT_i \bfV_i^\top$ and computing the next iterate by
\begin{equation}\label{eq:min_w_update}
    \bfx_{i+1} = \argmin_{\bfx} \frac{1}{2} (\bfx - \bfx_i)^\top \bfV_i \bfT_i \bfV_i^\top (\bfx - \bfx_i) + \nabla f (\bfx_i)^\top (\bfx - \bfx_i).
\end{equation}
Here, the columns of $\bfV_i \in \mathbb{R}^{n \times r}$ form an orthonormal basis for the Krylov subspace and $\bfT_i \in \mathbb{R}^{r \times r}$. Since $\nabla f(\bfx_i) \in {\rm row}(\bfJ) = {\rm col}(\bfJ^\top (\bfH_i + \beta_i \bfI) \bfJ)$ for $\beta_i>0$ and the Krylov subspace always contains $\nabla f(\bfx_i)$, the column space of $\bfV_i \bfT_i \bfV_i^\top$ always contains $\nabla f(\bfx_i)$. This means that the quadratic function~\eqref{eq:min_w_update} is bounded from below~\cite[Exercise~2.19]{beck2014introduction} and admits a minimum. The iterate $\bfx_{i+1}$ is the minimum norm solution to~\eqref{eq:min_w_update} given by
\begin{equation}\label{eq:update_dir}
    \bfx_{i+1} = \bfx_{i} + \Delta \bfx_{i}, \quad \text{where} \quad \Delta \bfx_{i} = - \bfV_i \bfT_i^{-1} \bfV_i^\top \nabla f(\bfx_i).
\end{equation}

Next, we state and prove some lemmas which will be used to prove the main theorem.
\begin{lemma}[Update Direction]\label{lemma:update_dir}
    The update $\Delta \bfx_i$ generated by the iterative scheme~\eqref{eq:update_dir} satisfies
    \begin{align}
        \Delta \bfx_i &\in {\rm row} (\bfJ) \label{eq:feature_space}, \\
        \Delta \bfx_i^\top \nabla^2 q_i(\bfx_i) \Delta \bfx_i &= \Delta \bfx_{i}^\top \bfV_{i} \bfT_{i} \bfV_{i}^\top \Delta \bfx_{i} \label{eq:full_rank_equi}.
    \end{align}
    Here,~\eqref{eq:feature_space} means that the update direction is in the row space of the linear model.
\end{lemma}

\begin{proof}[Proof of \Cref{lemma:update_dir}]
By construction, the Krylov subspace~\eqref{eq:Krylov} is a subspace of ${\rm row} (\bfJ)$, and by~\eqref{eq:update_dir} we have $\Delta \bfx_{i} \in {\rm col} (\bfV_{i})$. Thus we have $\Delta \bfx_{i} \in \text{col}(\bfV_{i}) \subseteq \text{row}(\bfJ)$, which proves~\eqref{eq:feature_space}.

Consider the full representation of the Hessian of~\eqref{eq:logsumexp_shift} generated by the Krylov subspace method
\begin{equation*}\label{eq:full_rank_repre}
    \nabla^2 q_{i}(\bfx_{i}) = \bfJ^\top (\bfH_{i} + \beta_{i} \bfI) \bfJ = \begin{bmatrix}
    \bfV_{i} & \bfU_{i}
    \end{bmatrix}
    \begin{bmatrix}
    \bfT_{i} & \bfD_1 \\
    \bfD_2 & \bfD_3
    \end{bmatrix}
    \begin{bmatrix}
    \bfV_{i}^\top \\
    \bfU_{i}^\top
    \end{bmatrix},
\end{equation*}
where ${\rm col}(\bfV_{i}) \perp {\rm col}(\bfU_{i})$. We have
\begin{align*} 
    \Delta \bfx_{i}^\top \nabla^2 q_i(\bfx_{i}) \Delta \bfx_{i} &= \Delta \bfx_{i}^\top \begin{bmatrix}
    \bfV_{i} & \bfU_{i}
    \end{bmatrix}
    \begin{bmatrix}
    \bfT_{i} & \bfD_1 \\
    \bfD_2 & \bfD_3
    \end{bmatrix}
    \begin{bmatrix}
    \bfV_{i}^\top \\
    \bfU_{i}^\top
    \end{bmatrix}
    \Delta \bfx_{i}  \\
    &= \begin{bmatrix}
    \Delta \bfx_{i}^\top \bfV_{i} & {\bf 0}
    \end{bmatrix}
    \begin{bmatrix}
    \bfT_{i} & \bfD_1 \\
    \bfD_2 & \bfD_3
    \end{bmatrix}
    \begin{bmatrix}
    \bfV_{i}^\top \Delta \bfx_{i} \\
    {\bf 0}
    \end{bmatrix}, \quad \text{as } \Delta \bfx_{i} \in {\rm col}(\bfV_{i}),  \\
    &=
    \Delta \bfx_{i}^\top \bfV_{i} \bfT_{i} \bfV_{i}^\top \Delta \bfx_{i}, 
\end{align*}
which proves~\eqref{eq:full_rank_equi}.
\end{proof}

\begin{lemma}[Descent Direction]\label{lemma:descent_cond}
The update $\Delta \bfx_{i}$ generated by~\eqref{eq:update_dir} satisfies the descent condition
\begin{equation}\label{eq:descent_cond}
\nabla f(\bfx_{i})^\top \Delta \bfx_{i} \leq - \Delta \bfx_{i}^\top \bfJ^\top (\bfH_{i} + \beta_{i} \bfI) \bfJ \Delta \bfx_{i}.
\end{equation}
\end{lemma}

\begin{proof}[Proof of \Cref{lemma:descent_cond}]
Since $\bfx_{i+1}$ is a solution to~\eqref{eq:min_w_update}, for any $t \in (0,1)$, we have
\begin{equation*}
    \frac{1}{2} \Delta \bfx_{i}^\top \bfV_{i} \bfT_{i} \bfV_{i}^\top \Delta \bfx_{i} + \nabla f(\bfx_{i})^\top \Delta \bfx_{i} \leq \frac{1}{2} (t\Delta \bfx_{i})^\top \bfV_{i} \bfT_{i} \bfV_{i}^\top (t \Delta \bfx_{i}) + \nabla f(\bfx_{i})^\top (t \Delta \bfx_{i}).
\end{equation*}
By rearranging the terms, we have
\begin{align*}
    \frac{(1-t^2)}{2} \Delta \bfx_{i}^\top \bfV_{i} \bfT_{i} \bfV_{i}^\top \Delta \bfx_{i} + (1-t) \nabla f(\bfx_{i})^\top \Delta \bfx_{i} & \leq 0 \\
    \frac{(1+t)}{2} \Delta \bfx_{i}^\top \bfV_{i} \bfT_{i} \bfV_{i}^\top \Delta \bfx_{i} + \nabla f(\bfx_{i})^\top \Delta \bfx_{i} & \leq 0 \\
    \nabla f(\bfx_{i})^\top \Delta \bfx_{i} & \leq - \frac{(1+t)}{2} \Delta \bfx_{i}^\top \bfV_{i} \bfT_{i} \bfV_{i}^\top \Delta \bfx_{i}.
\end{align*}
Letting $t \to 1^-$, we obtain
\begin{equation}\label{eq:low_rank_ineq}
    \nabla f(\bfx_{i})^\top \Delta \bfx_{i} \leq - \Delta \bfx_{i}^\top \bfV_{i} \bfT_{i} \bfV_{i}^\top \Delta \bfx_{i}.
\end{equation}
Combining~\eqref{eq:full_rank_equi} and~\eqref{eq:low_rank_ineq}, we obtain~\eqref{eq:descent_cond}.
\end{proof}

In the following lemma, we will make use of the fact that $\nabla f$ is Lipschitz continuous. This is because the gradient of the log-sum-exp function is the softmax function, which is Lipschitz continuous~\cite{gao2017properties,kong2020rankmax}.
\begin{lemma}[Armijo Line Search Condition]\label{lemma:armijo_CG}
Let $\lambda_{\rm min}$ be the smallest nonzero eigenvalue of $\bfJ^\top \bfJ$, and $L$ be the Lipschitz constant for $\nabla f$. For line search parameter $\gamma \in (0,1)$ and
\begin{equation}\label{eq:step_size_CG}
    \beta_i \geq \frac{L}{2\lambda_{\rm min}(1-\gamma)},
\end{equation}
the following Armijo line search condition holds
\begin{equation}\label{eq:armijo_CG}
    f(\bfx_{i+1}) \leq f(\bfx_{i}) + \gamma \nabla f(\bfx_i)^\top (\bfx_{i+1}-\bfx_i).
\end{equation}
\end{lemma}

\begin{proof}[Proof of Lemma~\ref{lemma:armijo_CG}]
First, note that
\begin{equation}\label{eq:min_eig_ineq_CG}
    \| \bfJ (\bfx_{i+1} - \bfx_{i})\|^2_{\bfH_i + \beta_i \bfI} \geq \beta_i \| \bfJ (\bfx_{i+1} - \bfx_{i}) \|_2^2 \geq \beta_i \lambda_{\rm min} \| (\bfx_{i+1} - \bfx_{i}) \|_2^2.
\end{equation}
Here, in the second step we used that $(\bfx_{i+1} - \bfx_{i}) \in \text{row}(\bfJ) = \text{row}(\bfJ^\top \bfJ)$ (Lemma~\ref{lemma:update_dir}), $\text{row}(\bfJ^\top \bfJ)^\perp = \text{null}(\bfJ^\top \bfJ)$, and $\lambda_{\rm min}$ is the smallest nonzero eigenvalue of $\bfJ^\top \bfJ$. 
Next, we have
\begin{align*}
    f(\bfx_{i+1}) & \leq f(\bfx_{i}) + \nabla f(\bfx_i)^\top (\bfx_{i+1} - \bfx_{i}) + \frac{L}{2} \| \bfx_{i+1} - \bfx_i \|_2^2 \\
    & \leq f(\bfx_{i}) + \nabla f(\bfx_i)^\top (\bfx_{i+1} - \bfx_{i}) + {\beta_i \lambda_{\rm min} (1-\gamma)} \| \bfx_{i+1} - \bfx_i \|_2^2 \\
    & \leq f(\bfx_{i}) + \nabla f(\bfx_i)^\top (\bfx_{i+1} - \bfx_{i}) + {(1 - \gamma)} \| \bfJ (\bfx_{i+1} - \bfx_{i})\|^2_{\bfH_i + \beta_i \bfI} \\
    & \leq f(\bfx_{i}) + \nabla f(\bfx_i)^\top (\bfx_{i+1} - \bfx_{i}) - (1 - \gamma) \nabla f(\bfx_i)^\top (\bfx_{i+1} - \bfx_i ) \\
    & = f(\bfx_{i}) + \gamma \nabla f(\bfx_i)^\top (\bfx_{i+1} - \bfx_{i}).
\end{align*}
Here, the first, second, thrid, and fourth steps use the Lipschitz continuity of $\nabla f$, \eqref{eq:step_size_CG}, \eqref{eq:min_eig_ineq_CG}, and Lemma~\ref{lemma:descent_cond}, respectively.
\end{proof}

\begin{lemma}[Stationary Point]\label{lemma:fixed_point}
The iterative scheme~\eqref{eq:update_dir} generates a fixed point $\bfx_*$ if and only if $\bfx_*$ is a stationary point.
\end{lemma}
\begin{proof}[Proof of~\Cref{lemma:fixed_point}]
"$\Leftarrow$": Substituting $\nabla f(\bfx_*) = {\bf 0}$ into~\eqref{eq:update_dir}, we obtain $\Delta \bfx_*=0$. Hence $\bfx_*$ is a fixed point.

\noindent "$\Rightarrow$": Let $\bfv = \bfx - \bfx_*$ for any $\bfx$. Since $\bfx_*$ is a fixed point to~\eqref{eq:min_w_update}, we have, for any $t \in \mathbb{R}$,
\begin{align*}
    & \quad \;  \frac{1}{2} (t\bfv)^\top \bfV_* \bfT_* \bfV_*^\top (t \bfv) + \nabla f(\bfx_*)^\top (t \bfv) \\
    & \geq \frac{1}{2} (\bfx_*-\bfx_*)^\top \bfV_* \bfT_* \bfV_*^\top (\bfx_*-\bfx_*) + \nabla f(\bfx_*)^\top (\bfx_*-\bfx_*).
\end{align*}
Simplifying this, we obtain
\begin{align*}
    \frac{t^2}{2} \bfv^\top \bfV_* \bfT_* \bfV_*^\top \bfv + t \nabla f(\bfx_*)^\top  \bfv & \geq {\bf 0} \nonumber \\
    \nabla f(\bfx_*)^\top  \bfv & \geq - \frac{t}{2} \bfv^\top \bfV_* \bfT_* \bfV_*^\top \bfv. 
\end{align*}
Taking $t \to 0$, we obtain $\nabla f(\bfx_*)^\top  \bfv \geq 0$ for any $\bfv$.
This implies $\nabla f(\bfx_*)$ is a zero vector, that is, $\bfx_*$ is a stationary point.
\end{proof}

Now, we are ready to prove the main theorem.
\begin{proof}[Proof of~\Cref{thm:main}]
The sequence $\{ f({\bfx_i}) \}_i$ is decreasing because the update directions are descent directions~(\Cref{lemma:descent_cond}) and the Armijo line search scheme guarantees sufficient descent at each step~(\Cref{lemma:armijo_CG}). By the continuity of $f$, it is closed~\cite[Proposition~1.1.2]{bertsekas2009convex}. Since $f$ is closed and attains its infimum in $\mathbb{R}$,
the decreasing sequence $\{ f({\bfx_i}) \}_i$ converges to a limit.

By the sufficient descent condition~\eqref{eq:armijo_CG}, the convergence of $\{ f({\bfx_i}) \}_i$ and $\alpha >0$,
\begin{equation*}
    \nabla f(\bfx_i)^\top (\bfx_{i+1} - \bfx_{i})
\end{equation*}
converges to zero. Hence, by~\eqref{eq:descent_cond}, 
\begin{equation*}
    \Delta \bfx_i^\top \bfJ^\top (\bfH_i + \beta_i \bfI) \bfJ \Delta \bfx_i 
\end{equation*}
converges to zero. Since $(\bfH_i + \beta_i \bfI)$ is positive definite and $\Delta \bfx_i \in \text{row}(\bfJ)$ (\Cref{lemma:update_dir}), $\Delta \bfx_i$ converges to the zero vector.

This implies that $\bfx_i$ converges to a fixed point of~\eqref{eq:update_dir}. By~\Cref{lemma:fixed_point}, $\bfx_i$ converges to a stationary point. By the convexity of $f$, $\bfx_i$ converges to a global minimum.
\end{proof}

\section{Numerical Experiments}\label{sec:experiments}
We perform two numerical experiments for minimizing the log-sum-exp function for a linear model. We compare the performance of the proposed LSEMINK with three commonly applied line search iterative methods and three disciplined convex programming (DCP) solvers; see~\Cref{subsec:benchmark}. In~\Cref{subsec:classification}, we consider multinomial logistic regression (MLR) arising in image classification. In~\Cref{subsec:LSE_experiment}, we experiment with a log-sum-exp minimization problem arising in geometric programming.
The experimental results show that LSEMINK has much better initial convergence, is more robust and scalable compared with the comparing methods.

\subsection{Benchmark Methods}\label{subsec:benchmark}
We compare the proposed LSEMINK with three common line search iterative schemes and three DCP solvers for machine learning and geometric programming applications. Firstly, we implement a standard Newton-CG (NCG) algorithm with a backtracking Armijo line search. Secondly, we compare with an $L^2$ natural gradient descent (NGD) method~\cite{NurbekyanEtAl2022,pascanu2013revisiting} that approximately solves
\begin{equation*}
    \min_{\bfx} \frac{1}{2} \nabla f (\bfx_i)^\top (\bfx - \bfx_i) + \frac{\lambda_i}{2} \sum_{k=1}^N w^{(k)} \| \bfJ^{(k)} (\bfx - \bfx_i)\|_2^2,
\end{equation*}
using CG to obtain the next iterate, where $\lambda_i$ controls the step size and is determined by a backtracking Armijo line search scheme, and the last term is a proximal term acting on the row space of the linear model. This scheme bears similarity to LSEMINK as the proximal term has the same effect as the shift in Hessian of LSEMINK. However, it does not make use of the Hessian and only approximates curvature information using the linear model. Thirdly, to demonstrate the effectiveness of the Hessian modification in LSEMINK, we compare with a standard modified Newton-Krylov (SMNK) scheme, which approximately solves~\eqref{eq:logsumexp_shift_general} with $\bfM_i=\bfI$ using Lanczos tridiagonalization, which has the same iterates as CG up to rounding errors but allows computations for the update direction to be re-used during line search. For LSEMINK, the Newton equation~\eqref{eq:Newton_eq} is approximately solved by CG. We note that an update direction has to be re-computed for each attempted value of $\beta_i$ during line search. In other words, unlike SMNK, the update direction computation cannot be re-used. However, our experimental results show that LSEMINK is still efficient in terms of computational cost thanks to the effectiveness of the modified Hessian. In each experiment, we use the same maximum number of iterations and tolerance for the CG and Lanczos schemes across different line search iterative methods.

In addition, we apply CVX~\cite{grant2008cvx}, a DCP package, paired with three different backend solvers (SPDT3~\cite{tutuncu2003solving}, SeDuMi~\cite{sturm1999using}, and MOSEK~\cite{aps2019mosek}). The best precision for CVX is used in the experiments; see~\cite{grant2008cvx} for detailed information.

\paragraph{Cost Measurement} We measure the computational costs for different line search iterative methods in terms of work units. In particular, a work unit represents a matrix-vector product with the linear models or their transpose. This is because these computations are usually the most expensive steps during optimization. For instance, in the MLR experiments of~\Cref{subsec:classification}, the linear models $\bfJ^{(k)}$'s contain the propagated high dimensional features of all the training data. Note that the number of work units in one iteration can differ across different line search iterative methods since a different number of CG/Lanczos iterations or line search updates can be performed. In addition to work unit, we also compare computational costs for all methods in total runtime.

\subsection{Experiment 1: Image Classification}\label{subsec:classification}
Perhaps the most prominent example of log-sum-exp minimization is multinomial logistic regression (MLR) arising in supervised classification. Here, we experiment on an MLR problem for the classification of MNIST~\cite{lecun1998mnist} and CIFAR-10~\cite{krizhevsky2009learning} image datasets. The MNIST dataset consists of $60,000$ $28 \times 28$ hand-written images for digits from 0 to 9. The CIFAR-10 consists of $60,000$ $32 \times 32$ color images equally distributed for the following ten classes: airplane, automobile,
bird, cat, deer, dog, frog, horse, ship, and truck. Example images for the two datasets are shown in~\Cref{fig:MNISTExample} and~\Cref{fig:CIFAR10Example}, respectively.

\begin{figure}[t]
    \centering
    \includegraphics[width=0.8\textwidth, height=1.3in]{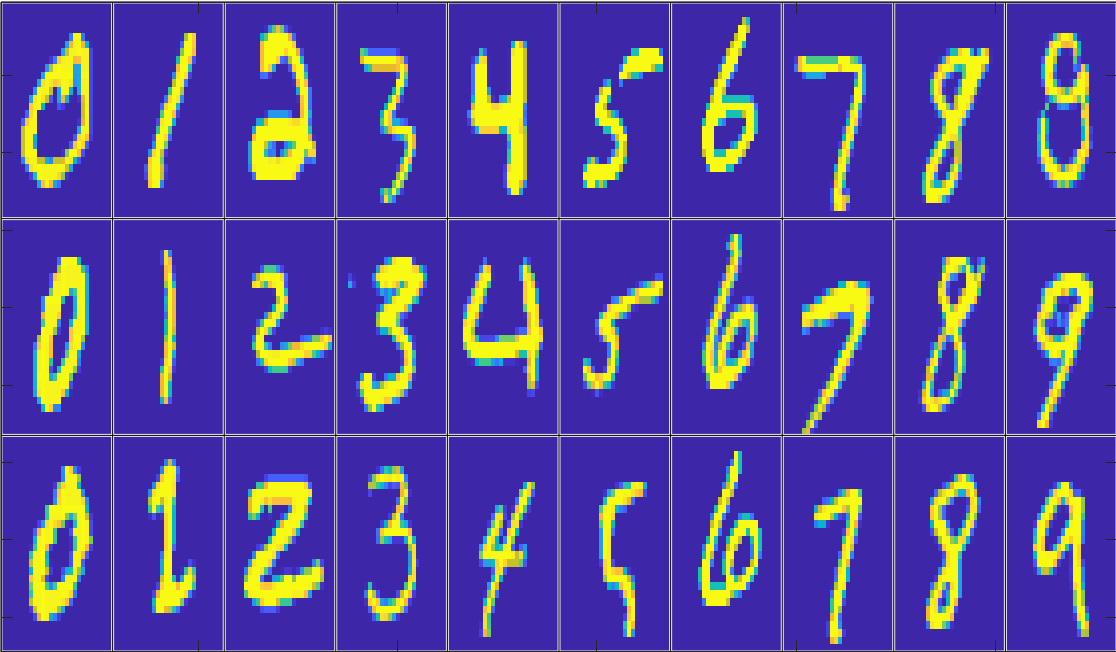}
    \caption{\small{Example images from the MNIST data set}}
    \label{fig:MNISTExample}
\end{figure}
\begin{figure}[t]
    \centering
    \begin{tabular}{cc}
    airplane
    \vspace{2.3mm} 
    &
    \multirow{9}{*}{\includegraphics[width=0.65\textwidth, height=2.75in]{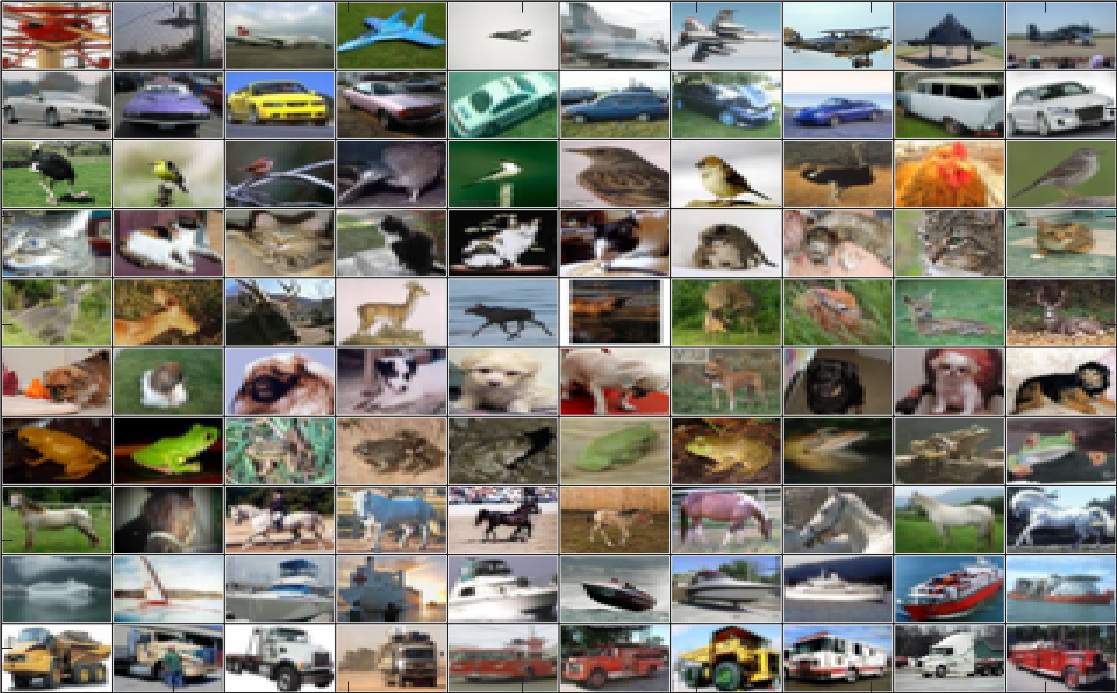}}
    \\
    automobile
    \vspace{2.3mm} &
    \\
    bird
    \vspace{2.3mm} &
    \\ 
    cat
    \vspace{2.3mm} & 
    \\
    deer
    \vspace{2.3mm} & 
    \\
    dog
    \vspace{2.3mm} & 
    \\
    frog
    \vspace{2.3mm} & 
    \\
    horse
    \vspace{2.3mm} &
    \\
    ship
    \vspace{2.3mm} & 
    \\
    truck
    \vspace{2.3mm} &
    \\
    \end{tabular}
    \label{CFAR10Imgs}
    \vspace{1em}
    \caption[Example images from the CIFAR-10 dataset]{{Example images for the CIFAR-10 dataset}\label{fig:CIFAR10Example}}
\end{figure}

\paragraph{Problem Description}
Let $n_f$ be the number of features, $n_c$ be the number of classes, and $\Delta_{n_c}$ be the $n_c$-dimensional unit simplex. Denote a set of data by $\{ \bfy^{(k)}, \bfc^{(k)} \}_{k=1}^N \subset \mathbb{R}^{n_f} \times \Delta_{n_c}$, where $\bfy^{(k)}$ and $\bfc^{(k)}$ are the input feature and target output label, respectively. In our experiments, we consider two feature extractors that enhance the features $\bfy^{(k)}$ by propagating it into a higher dimensional space $\mathbb{R}^{n_p}$. The first feature extractor is the random feature model (RFM)~\cite{huang2006extreme,rahimi2007random}. It applies a nonlinear transformation given by
\begin{equation*}
    \bfa_{\rm RFM}(\bfy^{(k)}) = \sigma(\bfZ \bfy^{(k)} + \bfb),
\end{equation*}
where $\sigma$ is the element-wise ReLU activation function, $\bfZ \in \mathbb{R}^{n_p \times n_f}$ and $\bfb \in \mathbb{R}^{n_p}$ are randomly generated. The second feature extractor is performed by propagating the features through the hidden layers of a pre-trained AlexNet~\cite{krizhevsky2017imagenet}. In particular, the AlexNet was pre-trained on the ImageNet dataset~\cite{deng2009imagenet}, which is similar to the CIFAR-10 dataset, using MATLAB's deep neural networks toolbox. This procedure is also known as transfer learning. These feature extractors can empirically enhance the generalization of the model, i.e., the ability to classify unseen data correctly.

The goal of the supervised classification problem is to train a softmax classifier
\begin{equation}\label{eq:softmax_classifier}
    s(\bfX, \bfa(\bfy^{(k)})) = \frac{ \exp(\bfX \bfa(\bfy^{(k)}))}{ {\bf 1}_{n_c} {\bf 1}_{n_c}^\top \exp(\bfX \bfa(\bfy^{(k)}))}
\end{equation}
such that $s(\bfX, \bfa(\bfy^{(k)})) \approx \bfc^{(k)}$. Here $\bfX$ are model parameters, the $\exp$ and division are applied element-wise, ${\bf 1}_{n_c}$ is an $n_c$-dimensional vector of all ones, and $\bfa: \mathbb{R}^{n_f} \to \mathbb{R}^{n_p}$ is a feature extractor.
To this end, we first consider the sample average approximation (SAA)~\cite{kleywegt2002,nemirovski2009,kim2015guide} of an MLR problem formulated as
\begin{align*}\label{eq:MLR_vector}
\begin{split}    
    \min_{\bfX \in \mathbb{R}^{n_c \times n_p}}  F(\bfX) &= - \frac{1}{N} \sum_{k=1}^N {\bfc^{(k)}}^\top \log \left( s(\bfX, \bfa(\bfy^{(k)})) \right) \\
    &=  \frac{1}{N} \sum_{k=1}^{N} \left[ ({\bfc^{(k)}}^\top {\bf 1}_{n_c}) \log \left({\bf 1}_{n_c}^\top \exp(\bfX \bfa(\bfy^{(k)})) \right) - {\bfc^{(k)}}^\top \bfX \bfa(\bfy^{(k)}) \right] \\
    &= \frac{1}{N} \sum_{k=1}^{N} \left[ \log \left({\bf 1}_{n_c}^\top \exp(\bfX \bfa(\bfy^{(k)})) \right) - {\bfc^{(k)}}^\top \bfX \bfa(\bfy^{(k)}) \right], 
\end{split}
\end{align*}
where the $\log$ operation is applied element-wise, and we use the fact that ${\bfc^{(k)}}^\top {\bf 1}_{n_c}=1$ since $\bfc^{(k)} \in \Delta_{n_c}$. 
The feature extractor is assumed to be fixed since the focus is on the log-sum-exp minimization problem. 
We vectorize the variable $\bfx = {\rm vec} (\bfX)$ so that the MLR problem becomes
\begin{align*}
    \min_{\bfx \in \mathbb{R}^{n_cn_p}} f(\bfx) =\frac{1}{N} \sum_{k=1}^{N} \left[ \log \left({\bf 1}_{n_c}^\top \exp(\bfJ^{(k)} \bfx) \right) - {\bfc^{(k)}}^\top \bfJ^{(k)} \bfx \right],
\end{align*}
which is of the form of~\eqref{eq:general_f} and where $\bfJ^{(k)} = \bfa(\bfy^{(k)})^\top \otimes \bfI_{n_c}$.

\paragraph{Experimental Results} In the MLR experiments, the line search iterative solvers stop when the norm of gradient is below $10^{-14}$ or after 3,000 work units. We stop the CG and Lanczos scheme when the norm of the relative residual drops below $10^{-3}$ or after 20 iterations. 

We first perform a small-scale experiment in which only $N=100$ training data is used, and a random feature model with dimension $m=1,000$ is applied. Since under this setup the data can be fit perfectly to achieve a zero training error, the model predictions~\eqref{eq:softmax_classifier} are close to standard basis vectors near an optimum. In this situation, the Hessian is close to a zero matrix, and the robustness of the solvers can be tested. The results are reported in~\Cref{table:ToyMLR} and~\Cref{fig:ToyMLR}. In~\Cref{table:ToyMLR}, one of the results for the standard Newton-CG scheme is not shown, as it fails to converge near the end. This is because the Hessian vanishes and consequently, the second-order approximation is unbounded from below. The natural gradient descent method has the slowest convergence and has yet to converge at the end. Both the standard modified Newton-Krylov method and LSEMINK achieve the stopping criteria under the specified work units. In particular, LSEMINK has superior convergence where the objective function value is up to five orders of magnitude smaller than the second-best method during optimization. LSEMINK also has the fastest time-to-solution. This demonstrates the effectiveness of LSEMINK and the efficacy of its modified Hessian over the standard one. SeDuMi, particularly SDPT3, can achieve very accurate results, but their runtime is about 15 times more than the LSEMINK. MOSEK fails to obtain a solution.

We then experiment with $n=50,000$ training data and $10,000$ validation data. For the MNIST dataset, we use an RFM to propagate the features to an $m=1,000$-dimensional space. For the CIFAR-10 dataset, features with dimension $m=9,216$ are extracted from the {\it{pool5}} layer of a pre-trained AlexNet. Here different feature extractors are used for the two datasets because a better validation accuracy can be achieved. In~\Cref{fig:MLR}, the results for an MLR problem are illustrated. In~\Cref{fig:MLR_reg}, we report the performance for an MLR problem with a Tikhonov regularization term $\frac{\alpha}{2} \| \bfx \|_{2}^2$, where $\alpha = 10^{-3}$. 
Using our state-of-the-art laptop, the CVX solvers cannot complete the experiments within thirty minutes, while the line search methods finish in thirty seconds.
Hence, we focus on the latter methods in this test. The figures show that the $L^2$ natural gradient descent method is the slowest. The standard Newton-CG and standard modified Newton-Krylov have good convergence results on one dataset but not the other. In contrast, LSEMINK is very competitive on both datasets. Specifically, it has good initial convergence where the objective function value is up to an order of magnitude smaller than the second-best scheme in the first few iterations. Moreover, its 
results are comparable with the other methods in terms of final training error, training accuracy, validation accuracy, and norm of gradient. 

\begin{table}[t]
\caption{Results on small-scale MLR experiments described in~\Cref{subsec:classification} in which the propagated random features have dimension $m=1,000$ and $N=100$ training data are used. The final objective function value, norm of gradient, and total runtime are reported. Some results are not shown because the corresponding scheme fails to return a solution. The tests are run on an Apple Macbook Pro with a 10-core M1 Max CPU and 32 GB of memory, and the software platform is MATLAB R2022a. \label{table:ToyMLR}}
\hspace{-4em}
\begin{tabular}{|c|c|c|c|c|c|c|c|c|}
\multicolumn{9}{c}{\hspace*{12pt}} \\
\hline
Dataset                  &                   & NCG      & NGD      & SMNK       & LSEMINK       & SeDuMi   & SDPT3     & MOSEK \\ \hline
\multirow{3}{*}{MNIST}   & $f$               & -- & 1.54e-02 & 1.41e-15 & 8.37e-16 & 6.65e-15 & 0.00e+00  & --   \\ \cline{2-9} 
                         & $\| \nabla f\|_2$ & -- & 1.33e-01 & 2.68e-15 & 7.05e-15 & 2.05e-14 & 5.14e-140 & --   \\ \cline{2-9} 
                         & Time              & --    & 2.58s    & 3.03s    & 1.70s    & 37.88s   & 28.51s    & --   \\ \hline
\multirow{3}{*}{CIFAR-10} & $f$               & 1.31e-15 & 1.27e-02 & 4.26e-15 & 8.77e-16 & 7.93e-15 & 0.00e+00  & --   \\ \cline{2-9} 
                         & $\| \nabla f\|_2$ & 6.16e-15 & 7.43e-02 & 6.58e-15 & 7.33e-15 & 2.11e-14 & 1.65e-212  & --   \\ \cline{2-9} 
                         & Time              & 1.95s    & 2.60s    & 3.02s    & 1.69s    & 31.60s   & 36.33s    & --   \\ \hline
\end{tabular}
\end{table}

\begin{figure}
\centering
\includegraphics[width=1\textwidth]{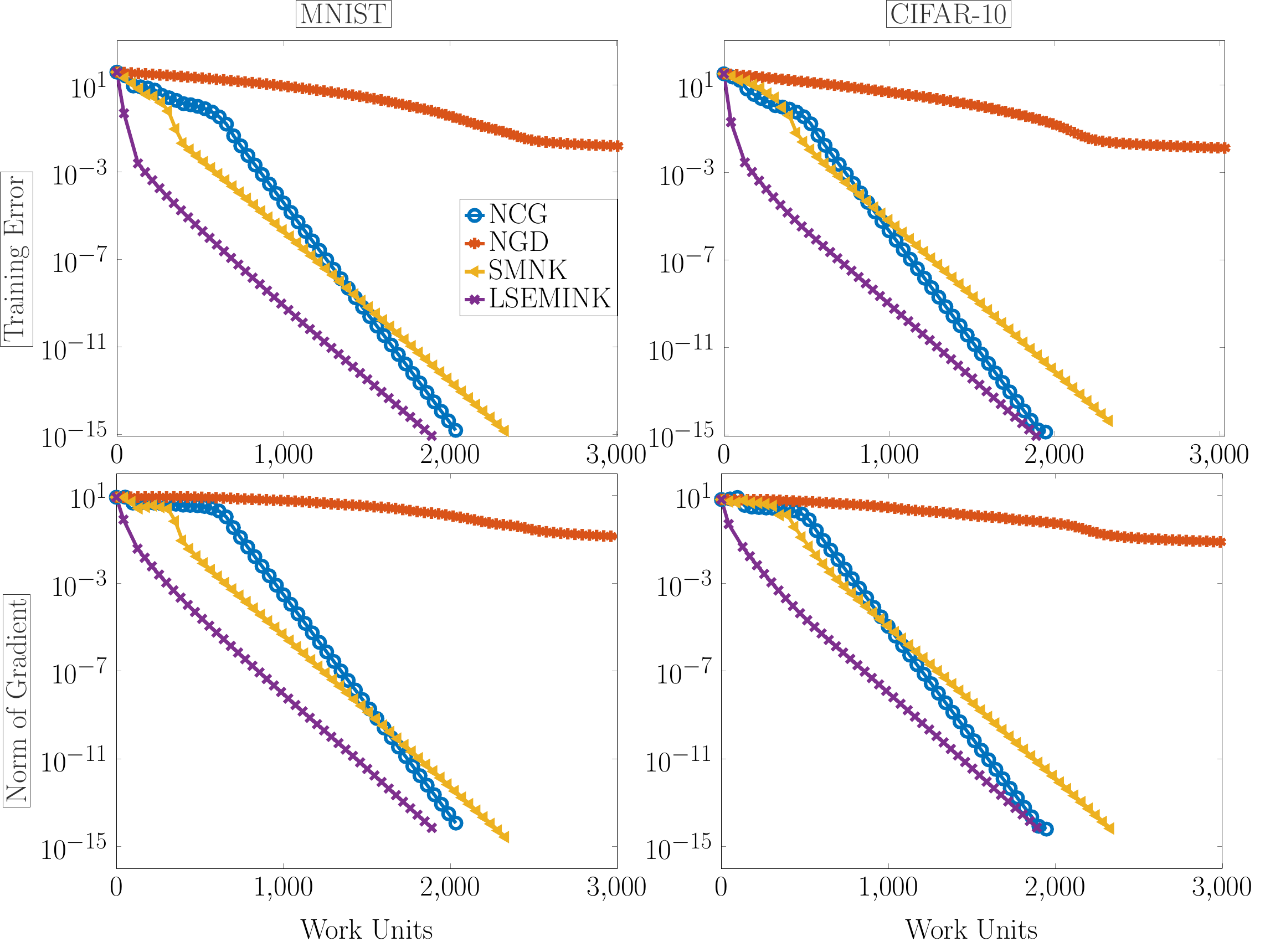}
\caption{Experimental results on small-scale MLR experiments in which the propagated random features have dimensions $n_p=1,000$ and $N=100$ training data are used.\label{fig:ToyMLR}}
\end{figure}	

\begin{figure}
\centering
\includegraphics[width=0.85\textwidth]{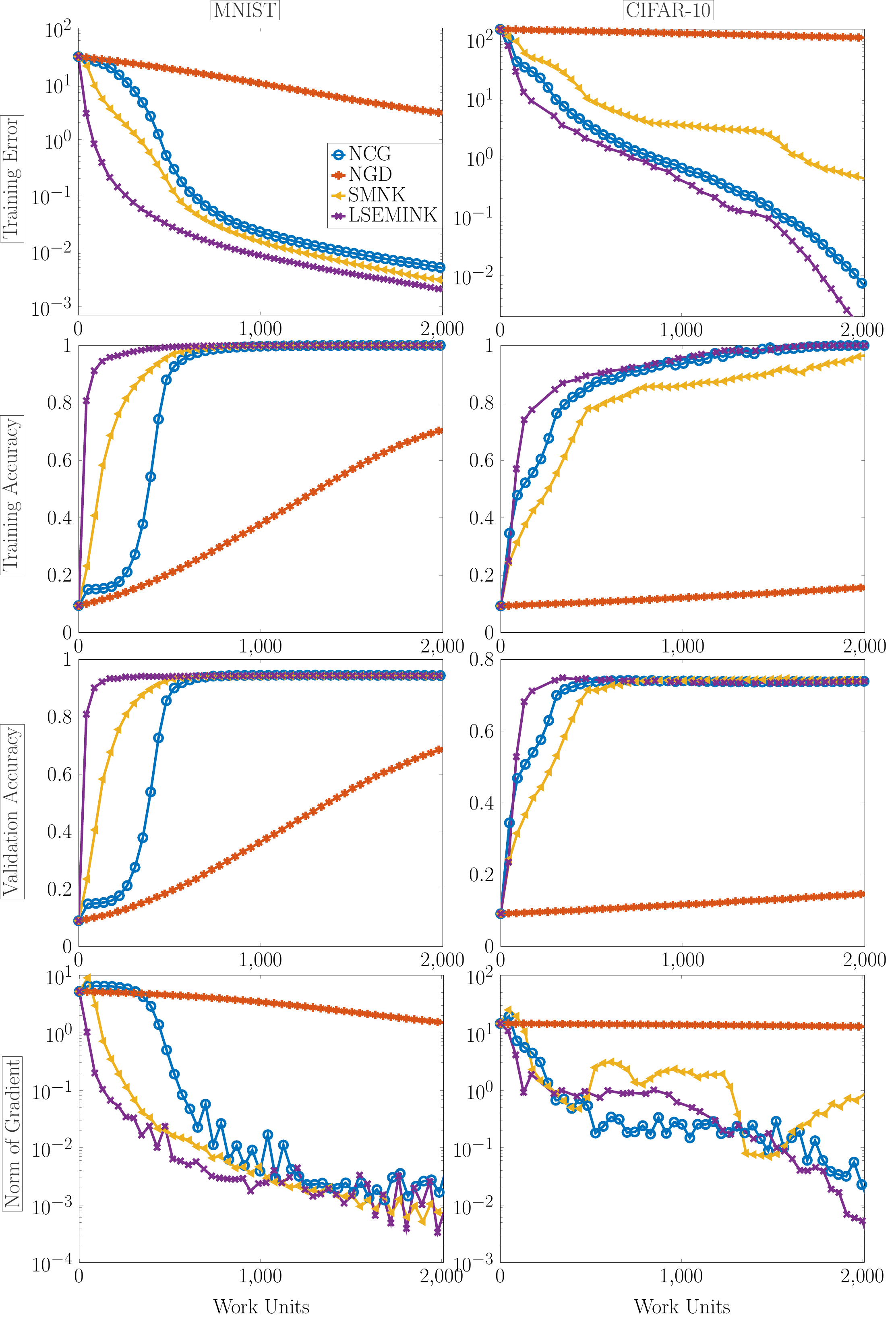}
\caption{Experimental results on MLR without regularization. The $x$-axes report the number of work units. Here $N=50,000$ training data and $10,000$ validation data are used.\label{fig:MLR}}
\end{figure}	

\begin{figure}
\centering
\includegraphics[width=0.85\textwidth]{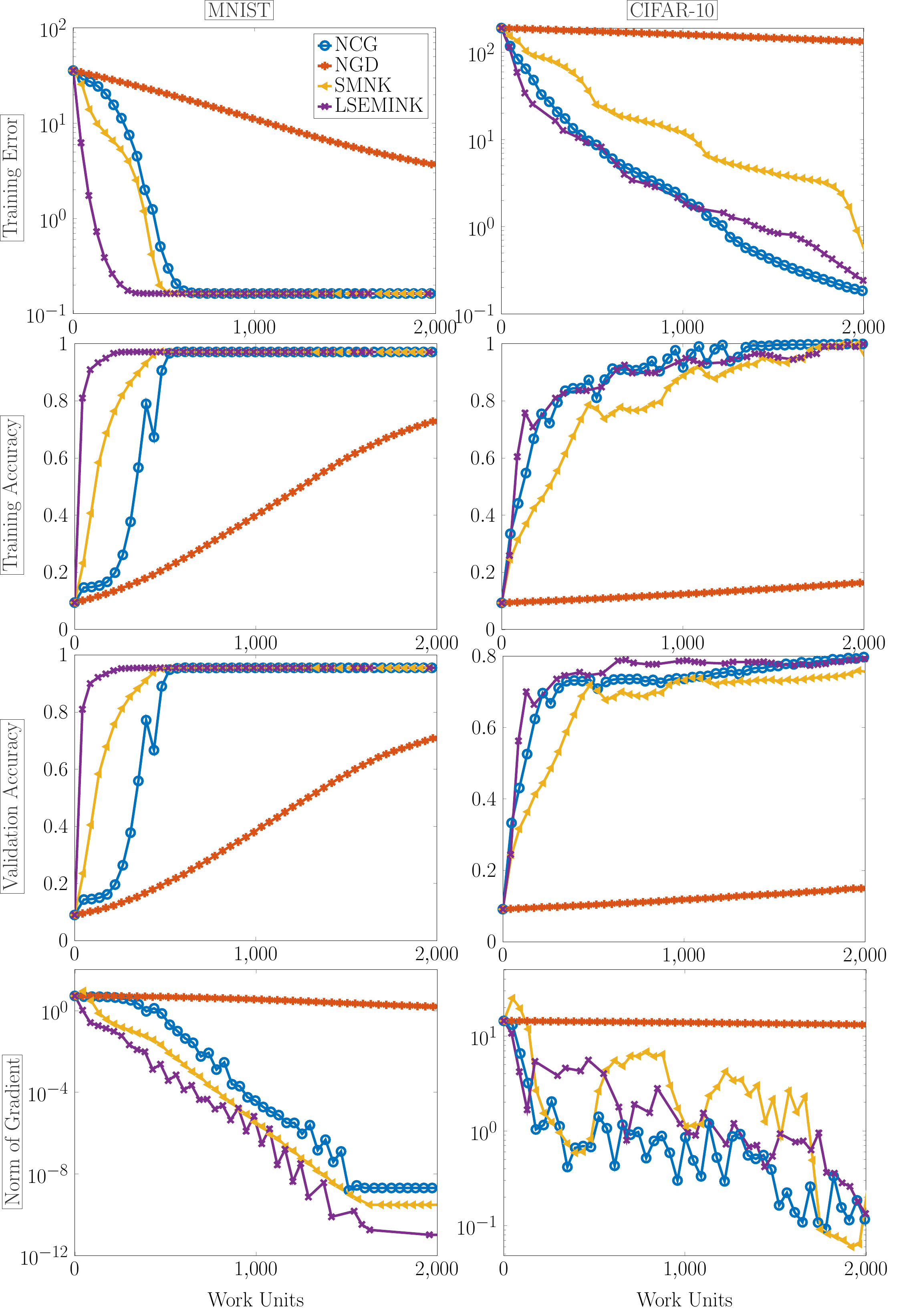}
\caption{Experimental results on MLR with a Tikhonov regularization $\frac{\alpha}{2} \| \bfx \|_2^2$, with $\alpha=10^{-3}$. The $x$-axes report the number of work units. Here $N=50,000$ training data and $10,000$ validation data are used. \label{fig:MLR_reg}}
\end{figure}	

\subsection{Experiment 2: Geometric Programming}\label{subsec:LSE_experiment}
We consider a log-sum-exp minimization problem which commonly arises in geometric programming~\cite{tseng2015milp,xi2020log,zhan2018accelerated} and is used to test optimization algorithms~\cite{kim2018adaptive,o2015adaptive}. In particular, it is formulated as
\begin{equation*}
    \min_{\bfx} {\eta} \log \left( {\bf 1}_m^\top \exp((\bfJ \bfx + \bfb)/\eta) \right),
\end{equation*}
where $\bfx \in \mathbb{R}^{n}$, $\bfJ \in \mathbb{R}^{m \times n}$, and $\eta$ controls the smoothness of the problem.
In particular, when $\eta \to 0$ the objective function converges to the point-wise maximum function $\max (\bfJ \bfx + \bfb)$ and its Hessian vanishes.

We follow the experimental setups in~\cite{kim2018adaptive,o2015adaptive}, which use $m=100$, $n=20$, and generate the entries of $\bfJ$ and $\bfb$ randomly. We perform the experiments with small values of $\eta$ to test the robustness of the methods. In particular, we test with $\eta=10^{-5}$,  $10^{-3}$, and $10^{-1}$, respectively. We stop the line search iterative schemes after $10,000$ work units. The CG and Lanczos schemes stop when the relative residual drops below $10^{-3}$ or after 20 iterations.

The experimental results are shown in~\Cref{table:LogSumExp} and~\Cref{fig:LogSumExp}. We see that the experiments are very challenging as the standard Newton-CG and all the CVX solvers cannot return a solution in some or all the experiments. In particular, the standard Newton-CG breaks in the first iteration in two of the experiments. This is because the quadratic approximation is unbounded from below. Both SeDuMi and SDPT3 fail in some of the experiments. MOSEK fails in all the experiments. When the CVX solvers succeed in returning a solution, they have significantly longer runtime (up to 60 times slower) compared to the line search methods. Similar to the previous experiments, $L^2$ natural gradient descent method has the slowest convergence and has yet to converge after the specified work units. The standard modified Newton-Krylov and LSEMINK are robust in the experiments and can return accurate solutions for $\eta=10^{-3}$ and $10^{-1}$. This indicates the effectiveness of Hessian modification in handling challenging optimization problems. Moreover, LSEMINK converges faster than the comparing standard modified Newton-Krylov method in the early stage. This indicates the effectiveness of the proposed Hessian modification over the standard one. However, we see that when $\eta=10^{-3}$ and $10^{-5}$, LSEMINK and all comparing methods cannot return a solution with the desired norm of gradient. This is because for a small $\eta$, the objective function is close to being nonsmooth. In contrast, the convergence of gradient based methods like LSEMINK requires the differentiability of the objective function.

\begin{table}[t]
\caption{Results on geometric programming experiments described in~\Cref{subsec:LSE_experiment}. The final objective function value, norm of gradient and total runtime are reported. Some results are not shown because the corresponding scheme fails to return a solution. The tests are run on an Apple Macbook Pro with a 10-core M1 Max CPU and 32 GB of memory, and the software platform is MATLAB R2022a. \label{table:LogSumExp}}
\hspace{-3em}
\begin{tabular}{|c|c|c|c|c|c|c|c|c|}
\multicolumn{9}{c}{\hspace*{12pt}} \\
\hline
$\eta$                &                   & NCG & NGD      & SMNK       & LSEMINK       & SeDuMi   & SDPT3    & MOSEK \\ \hline
\multirow{3}{*}{1e-5} & $f$               & -- & 7.48e+00 & 2.47e+00 & 1.44e+00 & 9.45e-01 & 9.45e-01      & --   \\ \cline{2-9} 
                      & $\| \nabla f\|_2$ & -- & 4.68e+00 & 5.96e+00 & 5.65e+00 & 5.76e-03 & 2.52e-06      & --   \\ \cline{2-9} 
                      & Time              & -- & 0.44s    & 1.43s    & 0.38s    & 2.18s    & 8.40s      & --   \\ \hline
\multirow{3}{*}{1e-3} & $f$               & -- & 7.37e+00 & 9.48e-01 & 9.48e-01 & -- & 9.48e-01 & --   \\ \cline{2-9} 
                      & $\| \nabla f\|_2$ & -- & 4.68e+00 & 3.80e-13 & 7.50e-11 & -- & 2.51e-06 & --   \\ \cline{2-9} 
                      & Time              & -- & 0.40s    & 0.80s    & 0.27s    & --    & 17.85s    & --   \\ \hline
\multirow{3}{*}{1e-1} & $f$               & 1.24e+00 & 2.43e+00 & 1.24e+00 & 1.24e+00 & --      & -- & --   \\ \cline{2-9} 
                      & $\| \nabla f\|_2$ & 2.72e-15 & 1.88e+00 & 2.38e-15 & 3.65e-15 & --      & -- & --   \\ \cline{2-9} 
                      & Time              & 0.09s & 0.35s    & 0.02s    & 0.02s    & --      & --   & --   \\ \hline
\end{tabular}
\end{table}

\begin{figure}
\centering
\includegraphics[width=1\textwidth]{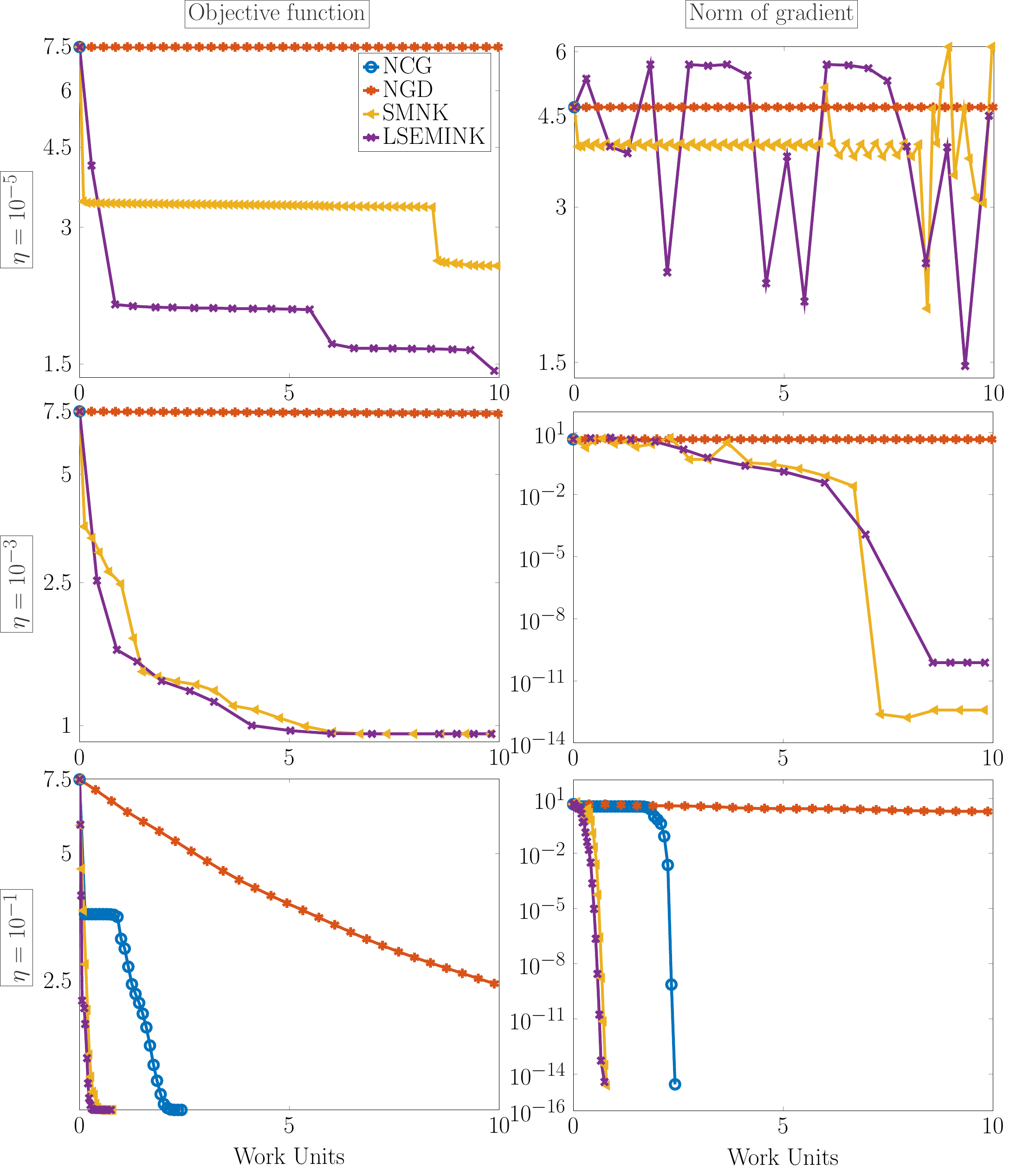}
\caption{Experimental results on geometric programming. The $x$-axes report the number of work units (in thousands). The results for the standard Newton-CG scheme are not shown because it fails in the first iteration.\label{fig:LogSumExp}}
\end{figure}	

\section{Conclusion}\label{sec:conclusion}
We present LSEMINK, a modified Newton-Krylov algorithm tailored for optimizing the log-sum-exp function for a linear model. The novelty of our approach is incorporating a Hessian shift in the row space of the linear model. This does not change the minimizers and renders the quadratic approximation to be bounded from below and the overall scheme to provably converge to a global minimum under standard assumptions. Since the update direction is computed using Krylov subspace methods which only require matrix-vector products with the linear model, LSEMINK is applicable to large-scale problems. Numerical experiments on image classification and geometric programming illustrate that LSEMINK has significantly faster initial convergence than standard Newton-Krylov methods, which is particularly attractive in applications like machine learning, and considerably reduces the time-to-solution and is more scalable compared to DCP solvers and natural gradient descent. Also, LSEMINK is more robust to ill-conditioning arising from the nonsmoothness of the problem. We provide a MATLAB implementation at~\url{https://github.com/KelvinKan/LSEMINK}. 

\section*{Acknowledgements}
This work was supported in part by NSF awards DMS 1751636, DMS 2038118, AFOSR grant FA9550-20-1-0372, and US DOE Office of Advanced Scientific Computing Research Field Work Proposal 20-023231. The authors would like to thank Samy Wu Fung for sharing the code for propagating the features of the CIFAR-10 dataset with AlexNet.

\bibliographystyle{abbrv}
\bibliography{main}
\end{document}